\documentclass[11pt,letterpaper]{amsart}
\usepackage{amsfonts}
\usepackage{bm}
\usepackage{bbm}  
\usepackage{fullpage}
\usepackage[top=2.50cm,left=2.50cm,right=2.50cm,bottom=3.15cm]{geometry}
\usepackage[colorlinks=true]{hyperref}
\usepackage[extra]{tipa}
\usepackage{indentfirst}
\usepackage{graphicx}

\usepackage{amsmath,amssymb,amsfonts,amsthm}
\usepackage[utf8]{inputenc}
\usepackage[english]{babel}
\usepackage{hyperref}
\usepackage{tikz-cd}
\usepackage{enumitem}
\usepackage{comment}

\newcommand\numberthis{\addtocounter{equation}{1}\tag{\theequation}}
\theoremstyle{definition}
\newtheorem{defn}{Definition}[section]
\newtheorem{theorem}[defn]{Theorem}

\newtheorem{lemma}[defn]{Lemma}
\newtheorem{prop}[defn]{Proposition}
\newtheorem{remark}[defn]{Remark}

\numberwithin{equation}{section}

\title[Explicit Reciprocity Laws for Formal Drinfeld Modules]{Explicit Reciprocity Laws for Formal Drinfeld Modules}
\author{Marwa Ala Eddine \\\\
\textit{U\MakeLowercase{niversit\'e de }B\MakeLowercase{ourgogne} F\MakeLowercase{ranche-}C\MakeLowercase{omt\'e,} L\MakeLowercase{aboratoire de} M\MakeLowercase{ath\'ematique de} B\MakeLowercase{esan\c con}, 25030 B\MakeLowercase{esan\c con,} F\MakeLowercase{rance.}} \\ \textit{E}\MakeLowercase{\textit{-mail: marwa.ala\_eddine@univ-fcomte.fr}},  \textit{ORCID: 0000-0001-6911-0580}
}

\begin{document}
\maketitle

\begin{abstract}
In this paper, we prove explicit reciprocity laws for a class of formal Drinfeld modules having stable reduction of height one, in the spirit of those existing in characteristic zero (cf. the work of Wiles \cite{Wiles}). We begin by defining  the Kummer pairing in the language of formal Drinfeld modules defined over local fields of positive characteristic. We then prove explicit formulas for this pairing in terms of the logarithm of the formal Drinfeld module, a certain Coleman power series, torsion points and the trace. Our results extend the explicit formulas already proved by Angl\`es \cite{angles} for Carlitz modules, and by Bars and Longhi \cite{Ignazio} for sign-normalized rank one Drinfeld modules. The approach followed is similar to the ones followed in the previously mentioned papers \cite{Wiles,angles,Ignazio}, taking into account the subtleties derived from the fact that the formal Drinfeld modules considered are formal power series, and are no longer polynomials.\\\\\footnotesize{ \text{Keywords:} formal Drinfeld modules, explicit reciprocity laws, local fields, class field theory} \\\footnotesize{\textit{2020 Mathematics Subject Classification: 11F85, 11S31}}
\end{abstract}

\footnote[0]{\textit{Statements and Declarations:} There is no conflict of interest.}


\section{Introduction} 

Explicit reciprocity laws have a long history. In 1928, Artin and Hasse \cite{Artin&Hasse} proved explicit formulas in characteristic zero for the multiplicative group. These formulas were completed by Iwasawa \cite{Iwasawa2} in 1968. In 1978, Wiles \cite{Wiles} proved an important generalization to the case of general Lubin-Tate formal groups. Soon after, Kolyvagin \cite{kolyvagin} extended all these results to formal groups of finite height. In the present paper, we place ourselves in positive characteristic, and we consider formal Drinfeld modules as defined by Rosen in \cite{Rosen}.  Let $K$ be a local field of positive characteristic. We know that formal Drinfeld modules can be seen as homomorphisms from the valuation ring of $K$ to the endomorphism ring of the formal additive group. Moreover, torsion points of such modules generate abelian towers of $K$. The maximal abelian extension of $K$ is equal to the compositum of the maximal unramified extension of $K$ and the union of these abelian towers. Therefore, there should be an analogue of the reciprocity laws in our settings. In \cite{angles}, Bruno Angl\`es considered a special class of formal Drinfeld modules, which he called Carlitz polynomials, and for which he proved explicit reciprocity laws in the spirit of those proved in characteristic zero. Later in 2009, Francesc Bars and Ignacio Longhi \cite{Ignazio} proved similar formulas for formal Drinfeld modules obtained from sign-normalized rank 1 Drinfeld modules.

Let $p$ be the characteristic of the local field $K$, and let $\mu_K$ be its normalized discrete valuation. We denote $\mathcal{O}$ the valuation ring of $K$ and $\mathfrak{p}$ its maximal ideal. Let $q$ be the order of the residue field $\mathcal{O}/\mathfrak{p}$. Then $q$ is a power of $p$. Fix an algebraic closure $\Omega$ of $K$, and let $\mu$ be the unique extension of $\mu_K$ to $\Omega$. Let $(\Bar{\Omega},\Bar{\mu})$ be the completion of $(\Omega,\mu)$. All the extensions $F$ of $K$ considered in this paper are supposed to be such that $F \subset \Omega$. 
We also denote $\mathcal{O}_F$ the valuation ring of $F$ and $\mathfrak{p}_F$ its maximal ideal.  Let $\pi$ be a fixed prime of $K$.

Let $B$ be an $\mathcal{O}$- algebra and let $\gamma: \mathcal{O} \longrightarrow B$ be the structure map. Let $B\{\{\tau \}\}$ be the twisted power series ring where $\tau$ is the $q$-Frobenius element satisfying 
\begin{equation}
    \tau x= x^q \tau, ~~ \forall x \in \Omega.
\end{equation}

Let $D: B\{\{\tau \}\} \longrightarrow B $ be the ring homomorphism that assigns to a power series $\sum_{n \ge 0}b_n \tau^n$ its constant term $b_0$. In \cite{Rosen}, Rosen defined a formal Drinfeld $\mathcal{O}$-module over $B $ to be a ring homomorphism 
\begin{align*} 
    \rho: \mathcal{O} & \longrightarrow B\{\{\tau \}\} \\
    a &  \mapsto \rho_a
\end{align*}
satisfying 
\begin{enumerate}[label=(\roman*)]
    \item $\forall a \in \mathcal{O}$, $D(\rho_a)=\gamma(a)$.
    \item $\rho(\mathcal{O}) \not\subset B $.
    \item $\rho_{\pi} \ne 0$.
\end{enumerate}

This definition is a special case of formal $\mathcal{O}$-modules over $B$ defined by Drinfeld in \S 1 of \cite{Drinfeld}. 
Let $f=\sum_{i \ge 0} b_n \tau^n \in B\{\{\tau \}\}$. We set $\operatorname{ord}_{\tau} (f)$ to be the least integer $n$ such that $b_n \ne 0$. The height of $\rho$ is defined by $\operatorname{ht}(\rho)= \operatorname{ord}_{\tau} (\rho_{\pi}) $. Clearly, the height is independent of the choice of the prime $\pi$. 

For any extension $K \subset L \subset \Bar{\Omega}$, the rings $\mathcal{O}_L$ and $\mathcal{O}_L / \mathfrak{p}_L$ are naturally $\mathcal{O}$-algebras. The structure map $\gamma : \mathcal{O} \longrightarrow \mathcal{O}_L$ is the inclusion map. Let $\rho$ be a formal Drinfeld $\mathcal{O}$- module over $\mathcal{O}_L$ as defined above. We say that $\rho$ has stable reduction if the ring homomorphism $\Bar{{\rho}}: \mathcal{O} \longrightarrow \mathcal{O}_L / \mathfrak{p}_L\{\{\tau \}\} $, obtained by reducing modulo $ \mathfrak{p}_L$ the coefficients of $\rho_a$, for $a \in \mathcal{O}$, is also a formal Drinfeld module. 

Let $K_{ur} \subset \Omega$ be the maximal unramified extension of $K$ in $\Omega$, and $H \subset K_{ur}$ be a finite unramified extension of $K$. Let $\rho$ be a formal Drinfeld $\mathcal{O}$-module over $\mathcal{O}_H$, having stable reduction, and such that $\operatorname{ht}(\Bar{\rho})=1$, then $\Bar{\Omega}$ is an $\mathcal{O}$-module for the  following action of $\rho$

\begin{equation} \label{defactionrho}
    a \cdot_{\rho} x =\rho_a(x) ~~ \forall x \in \Bar{\Omega}.
\end{equation} 

For an integer $n \ge 0$, let 

$$    V_{\rho}^n= \{ \alpha \in \Bar{\Omega} ;~ \rho_{a}(\alpha)=0 ~~\forall a \in \mathfrak{p}^n \} $$

\noindent be the $\mathfrak{p}^n$ torsion submodule of $\Bar{\Omega}$ for the action (\ref{defactionrho}).
Using the Weierstrass preparation theorem, we can see that $V_{\rho}^n \setminus V_{\rho}^{n-1}$ is the set of roots of a separable Eisenstein polynomial in $\mathcal{O}_H[X]$ of degree $q^{n-1}(q-1)$, whose constant term is a prime of $H$. Therefore, for an element $v_0 \in V_{\rho}^n \setminus V_{\rho}^{n-1}$, the extension $H(v_0) | H$ is totally ramified of degree $q^{n-1}(q-1)$. Furthermore, the kernel of $a \mapsto \rho_a(v_0)$ is $\mathfrak{p}^n$. Thus it induces an isomorphism of $\mathcal{O}$- modules 

\begin{equation} \label{isomWrho_n}
    \mathcal{O} / \mathfrak{p}^n \cong V_{\rho}^n.
\end{equation}

This implies that any element $v_0 \in V_{\rho}^n \setminus V_{\rho}^{n-1}$ is a generator of $V_{\rho}^n$ as $\mathcal{O}$-module. This also implies that the extension $H_{\rho}^n=H(V_{\rho}^n)$ is equal to $H(v_0)$. For more details see \cite{Iwasawa,Oukhaba}.
Now let $m_0 $ be an integer dividing $[H:K]$, and $\eta \in K$ of valuation $\mu(\eta)=m_0$. 
Let 

$$
     W_{\rho}^n=V_{\rho}^{nm_0} =\{ \alpha \in \mathfrak{p}_{\Bar{\Omega}} ; ~\rho_{\eta^n}(\alpha)=0\},~~ \text{and} ~~ W_{\rho}= \bigcup_n V_{\rho}^n =\bigcup_n W_{\rho}^n.
$$

 Let $$E_{\rho}^n= H(W_{\rho}^n)=H_{\rho}^{nm_0}.$$ 
 
\noindent Let $\mathcal{O}_n$ be the valuation ring of $E_{\rho}^n$ and $\mathfrak{p}_n$ be its maximal ideal. If $L$ is a finite extension of $E_{\rho}^n$, then we denote by
$$\Phi_L: L^{\times} \rightarrow Gal(L^{ab}|L) $$
the norm residue map. For an $\alpha \in \mathfrak{p}_{L}$ we will show in Section \ref{pairing1} that there exists $\xi \in L^{ab}$ such that $\rho_{\eta^n}(\xi)=\alpha$. Therefore we can define the map $(~,~)_{\rho,L,n}:\mathfrak{p}_L \times L^{\times} \longrightarrow W_{\rho}^n$ such that

\begin{equation} \label{eqdefpairing}
    (\alpha,\beta)_{\rho,L,n}=\Phi_L(\beta)(\xi)-\xi; ~~\rho_{\eta^n}(\xi)=\alpha.
\end{equation}

\noindent for $\alpha \in \mathfrak{p}_{L}$ and $\beta \in L^{\times}$. It is clear that $(~,~)_{\rho,L,n}$ is a bilinear from. \newline


The main objective of this paper is to prove explicit reciprocity laws for formal Drinfeld modules having stable reduction of height 1. In other words, we prove explicit formulas for the map $(~,~)_{\rho,L,n}$. 
Now we can state our main results.


\begin{prop} (Proposition \ref{WilesProp}) \label{iwasawamap}
Fix a generator $v_n$ of $ W_{\rho}^n$ as an $\mathcal{O}$-module and suppose $L | K$ is separable. 
There exists a unique map $\psi_{L,n}: L^n \longrightarrow \mathfrak{X}_L/\eta^{n+1}\mathfrak{X}_L$ 
such that 
\begin{equation} 
    (\alpha,\beta)_{\rho,L,n}=  T_{L|K}(\lambda_{\rho}(\alpha)\psi_{L,n}(\beta)) \cdot_{\rho}v_n
\end{equation}

\noindent for all $\alpha \in \mathfrak{p}_{L}$ and $\beta \in L^n$, where $\lambda_{\rho} $ is the logarithm of $\rho$, $ L^n=\{ \beta \in L^{\times};~ (\alpha,\beta)_{\rho,L,n}=0 ~ \forall \alpha \in  L \cap W_{\rho}\} $ and $\mathfrak{X}_L =\{ y \in L; ~\operatorname{T}_{L|K}(xy) \in \mathcal{O}~ \forall x \in \lambda_{\rho}(\mathfrak{p}_L)\}$.
\end{prop}


Proposition \ref{iwasawamap} is the analogue of Proposition 14 of \cite{Iwasawa2}. The map $\psi_{L,n}$ is the so-called Iwasawa map introduced in loc. cit. In the case where $L=E_{\rho}^n$ and $\rho$ is such that $ \rho_{\eta} \equiv \tau^{m_0} \mod \mathfrak{p}_H$, we can give an explicit form of $\psi_{L,n}$ in the following theorem.

\begin{theorem} (Theorem \ref{mainThm})
Suppose that $L=E_{\rho}^n$ and that $ \rho_{\eta} \equiv \tau^{m_0} \mod \mathfrak{p}_H.$ This means that if $\rho_{\eta}=\sum b_i \tau^i$, all the coefficients $b_i$ are in $\mathfrak{p}_H$ except for $b_{m_0}$.
Let $\alpha \in \mathfrak{p}_L$ such that $\mu(\alpha) \ge \frac{nm_0}{q}+\frac{1}{q-1}+ \frac{1}{q^{nm_0}(q-1)}$. Then for all $\beta \in L^{\times}$, we have 
\begin{equation} \label{eqresult}
     (\alpha,\beta)_{\rho,L,n}= \frac{1}{\eta^n} T_{L|K}(\lambda_{\rho}(\alpha)\delta_n(\beta)) \cdot_{\rho} v_n,
\end{equation}
where $\delta_n: L^{\times} \to \mathfrak{p}_L/\mathcal{D}_n$ is a group homomorphism defined as follows:
for $\beta \in L^{\times}$, choose a power series $f(X) \in \mathcal{O}_H((X))^{\times}$ such that $f(v_n)=\beta$, and set
\begin{equation} \label{deltadef}
    \delta_n(\beta):= \dfrac{f'(v_n)}{\beta} \mod \mathcal{D}_n.
\end{equation} 
For more details, see Lemma \ref{propertiesDelta} and Lemma \ref{isdefinedP2} .
\end{theorem}
Let $m \ge n$, let $\alpha \in \mathfrak{p}_n$ and $\alpha_m =\rho_{\eta^{m-n}}(\alpha)$. Let $\beta_m \in E_{\rho}^m$ and $\beta=\operatorname{N}_{m,n}(\beta_m)$, where $\operatorname{N}_{m,n}$ is the norm of the extension $E_{\rho}^m|E_{\rho}^n$, then 
\begin{equation} \label{eqm=n}
(\alpha_m,\beta_m)_{\rho,E_{\rho}^m, m}=  \Phi_{E_{\rho}^m}(\beta_m)(\xi)-\xi =  \Phi_{E_{\rho}^n}(\operatorname{N}_{m,n}(\beta_m))(\xi)-\xi =(\alpha,\beta)_{\rho,E_{\rho}^n, n},
\end{equation}
where $\xi$ is a root of $\rho_{\eta^n}(X)=\alpha$, hence a root of $\rho_{\eta^m}(X)=\alpha_m$. As a consequence of this equality, we deduce that \eqref{eqresult} is also valid for all $\alpha \in \mathfrak{p}_n$ and all $\beta \in \operatorname{N}_{m,n}(E_{\rho}^m)$ for $m \ge \dfrac{q}{q-1}(2n+\dfrac{1}{2m_0})$. This recalls both Theorem 19 of Wiles \cite{Wiles} and Theorem 3.12 of Angl\`es \cite{angles}. This also implies \eqref{eqresult} for all $\alpha \in \mathfrak{p}_n$ if $\beta$ is a universal norm in $(E_{\rho}^n)^{\times}$, which is the analogue of Theorem 1 of Wiles \cite{Wiles}.
Let us consider the inverse limit $\varprojlim (E_{\rho}^n)^{\times}$ with respect to the norm maps,
and the direct limit $\varinjlim \mathfrak{p}_n$ with respect to the maps
\begin{align*}
    \mathfrak{p}_n & \to \mathfrak{p}_m \numberthis  \\
    \alpha_n & \mapsto \rho_{\eta^{m-n}}(\alpha_n).
\end{align*}
We can define a limit form of $(~,~)_{\rho,L,n}$ as follows

\begin{equation} \label{defpairinglim}
(\alpha,\beta)_{\rho}=(\alpha_n,\beta_n)_{\rho, E_{\rho}^n, n}    
\end{equation}
for sufficiently large $n$, where $\alpha=(\alpha_n)_n \in \varinjlim \mathfrak{p}_n$ and $\beta =(\beta_n)_n \in \varprojlim E_{\rho}^n$. The limit form \eqref{defpairinglim} is well defined due to \eqref{eqm=n}. Moreover, we deduce from the discussion above that for all $\alpha=(\alpha_n)_n \in \varinjlim \mathfrak{p}_n$ and $\beta =(\beta_n)_n \in \varprojlim E_{\rho}^n$, we have 

\begin{equation}
(\alpha,\beta)_{\rho}=(\alpha_n,\beta_n)_{\rho, E_{\rho}^n, n} = \frac{1}{\eta^n} T_{L|K}(\lambda_{\rho}(\alpha_n)\delta_n(\beta_n)) \cdot_{\rho} v_n
\end{equation}
for sufficiently large $n$. Here, $\delta_n(\beta_n)$ can be expressed using the Coleman power series associated to $\beta \in \varprojlim E_{\rho}^n$. The existence of such power series was proved by Oukhaba in \cite{Oukhaba}. This gives a generalization of  Theorem 23 of Longhi-Bars \cite{Ignazio} proved for formal Drinfeld modules obtained from sign-normalized rank 1 Drinfeld modules. To go further, one may ask if any explicit reciprocity laws can be proved for all formal Drinfeld modules having stable reduction of height 1. We plan to address this question in a future work. In another request, we are interested in considering local fields of higher dimension in the vein of the work of Jorge Florez  \cite{Florez2,Florez1}.
\section{The Kummer pairing and first properties} \label{pairing1}
In this section, we fix a positive integer $n$ and a finite extension $L$ of $E_{\rho}^n$. In particular, we have $W_{\rho}^n \subset L$.
\begin{lemma} \label{existenceXi}
Let $\alpha \in \mathfrak{p}_L$. There exists an element $\xi$ in $\mathfrak{p}_{\Omega}$ such that $\rho_{\eta^n}(\xi)=\alpha$. Moreover, the extension $L(\xi)|L$ is abelian, of degree $\le q^{nm_0}$, and independent of the choice of $\xi$ satisfying $\rho_{\eta^n}(\xi)=\alpha$.
\end{lemma}

\begin{proof}
By \S 2 of \cite{Oukhaba}, we can write $\rho_{\pi^{nm_0}}$ as 
\begin{equation}
    \rho_{\pi^{nm_0}} = U_1U_{nm_0} Q_{nm_0} Q_{nm_0-1} \cdots Q_1,
\end{equation}
where $U_i$ are invertible elements of $\mathcal{O}_H\{\{\tau\}\}$ and $Q_i = \tau + \pi_i$, each $\pi_i$ being a prime of $H$. Let 
\begin{equation} \label{defPn}
  P_{nm_0}=Q_{nm_0} Q_{nm_0-1} \cdots Q_1  
\end{equation}
then $W_{\rho}^n$ is the set of roots of $P_{nm_0}(X)$.
We denote $V_n=\rho_{u^n}U_1 U_{nm_0}$. Since $V_n$ is invertible in $\mathcal{O}_H\{\{\tau\}\}$, we have
\begin{align*}
    \rho_{\eta^n}(X)=\alpha & \iff V_n(P_{nm_0}(X))=\alpha \\
                            & \iff P_{nm_0}(X) = V_n^{-1}(\alpha) \\
                            & \iff P_{nm_0}(X) - V_n^{-1}(\alpha) =0.
\end{align*}
However, $V_n^{-1}(\alpha) \in \mathfrak{p}_L$, hence, $P_{nm_0}(X)-V_n^{-1}(\alpha)$ is a polynomial with coefficients in $L$. Therefore there exists an element $\xi$ in ${\Omega}$ such that $P_{nm_0}(\xi)-V_n^{-1}(\alpha)=0 $. Furthermore, since $0 \equiv P_{nm_0}(\xi) \equiv \xi^{q^{nm_0}} \mod \mathfrak{p}_{\Omega}$, we have $\xi \in \mathfrak{p}_{\Omega}$.
Moreover, the polynomial $P_{nm_0}(X)-V_n^{-1}(\alpha)$ is of degree $q^{nm_0}$, and all the elements of the set $\xi + W_{\rho}^n$, which we recall is a set of $q^{nm_0}$ elements, are roots of this polynomial. Hence, it is separable and $L(\xi)|L$ is a Galois extension of degree $\le q^{nm_0}$ depending only on $\alpha$. Finally, to prove that it is an abelian extension, it suffices to notice that the group homomorphism $\operatorname{Gal}(L(\xi)|L) \longrightarrow W_{\rho}^n$ defined by $\sigma \mapsto \sigma(\xi)-\xi$ is injective.
\end{proof}         

By this Lemma, we see that the map $(~,~)_{\rho,L,n}: \mathfrak{p}_L \times L^{\times} \longrightarrow W_{\rho}^n$ in (\ref{eqdefpairing}) is well defined. We omit $\rho$ in the index when there is no risk of confusion. Exactly as in \cite{Wiles,kolyvagin} we have

\begin{prop} \label{PropoertiesPairing1}
The map $(~,~)_{L,n}$ satisfies the following properties

\begin{enumerate} [label=(\roman*)]
    \item The map $(~,~)_{L,n}$ is bilinear and $\mathcal{O}$-linear in the first coordinate for the action (\ref{defactionrho}).
    \item We have $$(\alpha,\beta)_{L,n}=0  \iff \beta  \text{ is a norm from } L(\xi)\text{, where } \rho_{\eta^n}(\xi)=\alpha.$$ 
    \item Let $M$ be a finite separable extension of $L$, let $\alpha \in \mathfrak{p}_L$ and $\beta \in M^{\times}$. Then $(\alpha,\beta )_{M,n} =(\alpha,\operatorname{N}_{M|L}(\beta))_{L,n}$.
    \item Let $M$ be a finite separable extension of $L$ of degree $d$, let $\alpha \in \mathfrak{p}_M$ and $\beta \in L^{\times}$. Then $(\alpha,\beta)_{M,n}=(\operatorname{T}_{M|L}(\alpha),\beta)_{L,n}$.
    \item Suppose $L \supset E_{\rho}^m$ for $m \ge n$. Then 
    $$  (\alpha,\beta)_{L,n} = \rho_{\eta^{m-n}}((\alpha,\beta)_{L,m} ) = (\rho_{\eta^{m-n}}(\alpha),\beta)_{L,m} .$$
    \item Let $\rho'$ be a formal Drinfeld $\mathcal{O}$-module isomorphic to $\rho$, i.e there exists a power series $t$ invertible in $\mathcal{O}_H\{\{\tau\}\}$ such that $\rho'_a= t^{-1} \circ \rho_a \circ t$ for all $a \in \mathcal{O}$. Then we have $(\alpha,\beta)_{\rho',L,n}= t^{-1}((t(\alpha),\beta)_{\rho,L,n})$.
\end{enumerate}
\end{prop}

\begin{proof}
Since (i), (ii), (iii), (v) and (vi) are straightforward, we only prove (iv). Let $\xi$ be such that $\rho_{\eta ^n}(\xi)=\alpha$. Let $G$ be the Galois group of $M^{ab}|L$ and $H$ be the Galois group of $M^{ab}|M$. Let $\{\sigma_1,..,\sigma_d\}$ be  the embeddings of $M$ in $\Omega$ over $L$. We consider the quotient group $G/H$. Then $\{\Tilde{\sigma_1},..,\Tilde{\sigma_d}\}$, where $\Tilde{\sigma_i}$ is an extension of $\sigma_i$ to $M^{ab}$, is a set of representatives for the left cosets $H \Tilde{\sigma}$
$$ G= \sqcup_i H \Tilde{\sigma_i}.$$
We extend $\Phi_L(\beta)$ to $M^{ab}$ and denote it by $\Tilde{\Phi}_L(\beta) \in G$. Therefore, for each $i$, there exists a unique $h_i \in H$, and a unique $j$ such that 
\begin{equation} \label{equivclass}
    \Tilde{\sigma_i} \Tilde{\Phi}_L(\beta) = h_i \Tilde{\sigma_j}.
\end{equation}
By definition of the transfer map $t_{M|L}$ (see for instance \cite[\S 6.3]{Iwasawa}), we have 

\begin{align} 
    (\alpha,\beta)_{M,n} & =\Phi_M(\beta)(\xi)-\xi \nonumber \\
    & = t_{M|L}(\Phi_L(\beta))(\xi)-\xi  \nonumber \\
    & = \prod_i h_i(\xi)-\xi \nonumber \\
    &= \sum_i (h_i(\xi)-\xi) \label{pairing_m}
\end{align}

\noindent The last equality follows from the fact that $h_i(\xi)-\xi \in W_{\rho}^n \subset L$. Using the notation  $(\alpha,h_i)_{M,n}=h_i(\xi)-\xi$, we can write
\begin{equation} \label{hisum}
   \sum_i (\alpha,h_i)_{M,n} =  \Tilde{\Phi}_L(\beta)(\sum_j(\Tilde{\sigma}_j^{-1}(\xi)))-\sum_j(\Tilde{\sigma}_j^{-1}(\xi)) .
\end{equation}

\noindent Indeed, by (\ref{equivclass}) we get

\begin{align*}
    \Tilde{\Phi}_L(\beta)\Tilde{\sigma}_j^{-1}(\xi) &= \Tilde{\sigma}_i^{-1}h_i(\xi) \\
    & = \Tilde{\sigma}_i^{-1} ((\alpha,h_i)_{M,n}+\xi) \\
    & = (\alpha,h_i)_{M,n} + \Tilde{\sigma}_i^{-1}(\xi)
\end{align*}
because $(\alpha,h_i)_{M,n} \in W_{\rho}^n \subset L$. Thus, by (\ref{pairing_m}) and (\ref{hisum}), we have $(\alpha,\beta)_{M,n}= (\sum_i \sigma_i(\alpha),\beta)_{L,n}$.
\end{proof}

\section{The Iwasawa map}
In this section, we will study the so-called Iwasawa map, first introduced by Iwasawa in \cite[Proposition 14]{Iwasawa2} in the cyclotomic case. This map was generalized by Wiles \cite[Proposition 7]{Wiles} in the case of Lubin-Tate formal groups, and by Kolyvagin \cite[Proposition 3.2]{kolyvagin} in the case of formal groups of finite height. As in Section \ref{pairing1}  above, we fix a positive integer $n$ and a finite extension $L$ of $E_{\rho}^n$. We also fix a generator $v_n$ of the $\mathcal{O}$-module $W_{\rho}^n$ and we suppose that $L|K$ is separable. First, we need to introduce the logarithm $\lambda_{\rho}$ of $\rho$.

\begin{lemma} \label{lemmadefrho}
There exists a unique power series $\lambda_{\rho} \in H\{\{\tau\}\}$, called the logarithm of $\rho$, such that $\lambda_{\rho}(X) \equiv X \mod \operatorname{deg~2}$ and $\lambda_{\rho} \rho_a = a \lambda_{\rho}$ for all $a \in \mathcal{O}$. Moreover, we have
\begin{enumerate}[label=(\roman*)]
    \item \label{coefflamda} If $\lambda_{\rho}=\sum_{i \ge 0} c_i \tau^i$, then $ \mu(c_i) \ge -i$ for all $i \ge 0$. Thus the element $\lambda_{\rho}(x)=\sum_{i \ge 0} c_i x^{q^i}$ is well defined in $L$ for any $x \in \mathfrak{p}_L$.
    \item If $x \in \mathfrak{p}_{\Omega}$, then $\lambda_{\rho}(X)=0$ if and only if $x \in V_{\rho}$. Put $W_L=L \cap W_{\rho} \subset \mathfrak{p}_L$. Then the map $\lambda_{\rho}: \mathfrak{p}_L/W_L \longrightarrow  \lambda_{\rho}(\mathfrak{p}_L)$ is an isomorphism of $\mathcal{O}$-modules.
    \item Let $\mathfrak{p}_{\Omega,1}$ denote the set of all the elements $x$ of $\mathfrak{p}_{\Omega}$ such that $\mu(x) > 1/(q-1)$. The logarithm $\lambda_{\rho}$ gives an isomorphism of $\mathcal{O}$-modules from $\mathfrak{p}_{\Omega,1}$, viewed as an $\mathcal{O}$-module under the action (\ref{defactionrho}), to itself, viewed as an $\mathcal{O}$-module under the multiplication in $\Omega$. If we denote $\mathfrak{p}_{L,1} = \mathfrak{p}_{L} \cap \mathfrak{p}_{\Omega,1}$, the logarithm $\lambda_{\rho}$ also induces an isomorphism from $\mathfrak{p}_{L,1}$ to itself.
    \item The $\mathcal{O}$-module $\lambda_{\rho}(\mathfrak{p}_L)$ is free of rank $[L:K]$ and we have  $L=K \lambda_{\rho}(\mathfrak{p}_L)$.
\end{enumerate}
\end{lemma}

\begin{proof}

    The first three properties are proved by M. Rosen in \cite{Rosen}. For instance, the property (i) is a part of the proof of Proposition 2.1 of loc. cit. The property (ii) is exactly Proposition 2.4 of \cite{Rosen}. Finally, (iii) corresponds to Proposition 2.3 of \cite{Rosen}. Let us  give a sketch of the proof of (iii).
    Let $x \in \mathfrak{p}_{\Omega}$ such that $\mu(x) > \frac{1}{q-1}.$ By \ref{coefflamda}, we have
    $$
        \mu(c_ix^{q^i})=\mu(c_i)+q^i \mu(x) \ge -i + q^i \mu(x) >\mu(x)
    $$
    for all $i \ge 1$. Hence $\mu(\lambda_{\rho}(x))=\mu(x)$ so that $\lambda_{\rho}(x) \in \mathfrak{p}_{\Omega,1}$. Now we consider the inverse $e_{\rho}$ of $\lambda_{\rho}$ in $H\{\{\tau\}\}$. This series is called the exponential of $\rho$ and satisfies $e_{\rho}(X) \equiv X \mod \operatorname{deg}~2$ and $e_{\rho}a = \rho_a e_{\rho}$ for all $a \in \mathcal{O}$. By \cite[Proposition 2.2]{Rosen}, if we write $e_{\rho}(x)=x+\sum_{i \ge 1}d_ix^{q^i}$, we have $\mu(d_i) \ge -(1+q+\cdots + q^{i-1})$. Thus,
    $$
    \mu(d_ix^{q^i})=\mu(d_i)+q^i \mu(x) \ge -\frac{q^i-1}{q-1} + q^i \mu(x) > -(q^i-1)\mu(x) +\mu(x)=\mu(x)
    $$
    for all $i \ge 1$. Hence we have $\mu(e_{\rho}(x))=\mu(x)$. This completes the proof since $e_{\rho}$ is the formal inverse of $\lambda_{\rho}$. 
    
    As for the proof of (iv), let $x \in \mathfrak{p}_L$ and $e_L$ be the ramification index of $L|K$, then $\mu(x) \ge \displaystyle\frac{1}{e_L}$. By \ref{coefflamda}, we have 
\begin{align*}
    \mu(\lambda_{\rho}(x)) & \ge \min(\mu(x), -i+q^i \mu(x)) \\
    & \ge \min (\frac{1}{e}, -i+\frac{q^i}{e}).
\end{align*}
Thus, for a sufficiently large integer $l$, we have $\lambda_{\rho}(\mathfrak{p}_L) \subset \displaystyle\frac{1}{\pi^l}\mathcal{O}_L$. Therefore $\lambda_{\rho}(\mathfrak{p}_L)$ is free for it is a $\mathcal{O}$-submodule of the free $\mathcal{O}$-module $\displaystyle\frac{1}{\pi^l}\mathcal{O}_L$. Now let us prove that $L=K \lambda_{\rho}(\mathfrak{p}_L)$.
Clearly, we have  $K \lambda_{\rho}(\mathfrak{p}_L) \subset L$. Let $x \in L$, then we can write $x= u \pi_L^j$, where $u$ is a unit of $L$ and $\pi_L$ is a prime of $L$.  Then, for a sufficiently large integer $i$, we have $u \pi_L^j \pi^i \in \mathfrak{p}_{L,1} = \lambda_{\rho}(\mathfrak{p}_{L,1}) \subset \lambda_{\rho}(\mathfrak{p}_L) $. Therefore $x= \displaystyle  \frac{1}{\pi^i} u \pi_L^j \pi^i \in K \lambda_{\rho}(\mathfrak{p}_L)$.
\end{proof}

Since the extension $L|K$ is supposed to be separable, the bilinear map $<~,~>_L:L \times L \longrightarrow K$ defined by $<x,y>_L=\operatorname{T}_{L|K}(xy)$ is non degenerate. This gives us the classical isomorphism from $L$ to the space of $K$-linear forms from $L$ to $K$.  The pairing $<~,~>_L$ also induces the following $\mathcal{O}$-linear map
\begin{align*}
    L & \longrightarrow \operatorname{Hom}_{\mathcal{O}} (\lambda_{\rho}(\mathfrak{p}_L)  , K/\mathcal{O} ) \numberthis \label{eqn} \\
    y & \mapsto \left\{ \begin{array}{cc}
        \lambda_{\rho}(\mathfrak{p}_L)  \longrightarrow & K/\mathcal{O} \\
       x  \longmapsto & <x,y>_L \mod \mathcal{O}
    \end{array}
    \right.
\end{align*}

\begin{lemma} \label{surjmap}
The map \eqref{eqn} is a surjective homomorphism of $\mathcal{O}$-modules, with kernel 
\begin{equation} \label{defXL}
     \mathfrak{X}_L :=\{ y \in L;~ <x,y>_L \in \mathcal{O}~ \forall x \in \lambda_{\rho}(\mathfrak{p}_L)\}.
\end{equation}
\end{lemma}

\begin{proof}
It is clear that $\mathfrak{X}_L$ is the kernel of this map. Let us prove that the map is surjective. To do so, let $\gamma:\lambda_{\rho}(\mathfrak{p}_L)\rightarrow K/\mathcal{O}$ be an $\mathcal{O}$-linear map. The $\mathcal{O}$-module $K/\mathcal{O}$ is divisible, hence it is an injective module  by \cite[Lemma 4.2]{Lang}. Therefore there exists a homomorphism of $\mathcal{O}$-modules $\Tilde{\gamma}$ such that the following diagram

\begin{center}
\begin{tikzcd}
L \arrow[rd, "\Tilde{\gamma}"]  &  \\
\lambda_{\rho}(\mathfrak{p}_L) \arrow[u] \arrow[r, "\gamma"'] & K/\mathcal{O}
\end{tikzcd}
\end{center}
is commutative, the left hand map being the inclusion map. Let $\{e_1,\ldots,e_d\}$ be a basis of $L$ as a $K$-vector space. Since $L=K\lambda_{\rho}(\mathfrak{p}_L)$ by Lemma \ref{lemmadefrho} (iv), we can choose the $e_i$ to be in $\lambda_{\rho}(\mathfrak{p}_L)$. Choose elements $\doubletilde{$\gamma$}(e_i)$ in $K$ such that $\tilde{\gamma}(e_i)$ is the class of $\doubletilde{$\gamma$}(e_i)$ modulo $\mathcal{O}$. Define the $K$-linear map $\doubletilde{$\gamma$}:L \longrightarrow K$ by $\doubletilde{$\gamma$}(\sum a_i e_i)=\sum a_i \doubletilde{$\gamma$}(e_i)$ where $a_i \in K$. Thus we obtain the following commutative diagram
\[
\begin{tikzcd}
L \arrow{r}{\doubletilde{$\gamma$}} \arrow{rd}{\tilde{\gamma}} & K \arrow{d} \\
\lambda_{\rho}(\mathfrak{p}_L) \arrow{u} \arrow{r}{\gamma} & K/\mathcal{O}
\end{tikzcd}
\]
\noindent the right hand arrow being the canonical projection and the left hand arrow being the inclusion.
However, the $K$-linear form $\doubletilde{$\gamma$}$ is induced by some element $y \in L$ satisfying $\doubletilde{$\gamma$}(x)=\operatorname{T}_{L|K}(xy)$ for all $x \in \lambda_{\rho}(\mathfrak{p}_L)$. Therefore we have $\gamma(x) \equiv \doubletilde{$\gamma$}(x) = <x,y>_L \mod \mathcal{O}$.
\end{proof}

Now, we give the construction of the so-called Iwasawa map. As mentioned in the introduction, the map 
\begin{align*}
\mathcal{O}/\eta^n \mathcal{O}  & \longrightarrow W_{\rho}^n  \numberthis\\
a  & \longmapsto \rho_a(v_n)
\end{align*}
is an isomorphism of $\mathcal{O}$-modules. We denote by $\iota_1$ its inverse. We define the $\mathcal{O}$-linear map
\begin{displaymath}
\begin{tikzcd}
    \iota: W_{\rho}^n  \ar[r, "\iota_1"] & \mathcal{O}/\eta^n \mathcal{O}  \ar[r] & K/\mathcal{O} \\
     \rho_a(v_n)  \ar[r, mapsto]  & a  \ar[r, mapsto] & \frac{a}{\eta^n}
\end{tikzcd}
\end{displaymath}
Let
\begin{equation} \label{defLn}
    L^n=\{ \beta \in L^{\times};~ (\alpha,\beta)_{L,n}=0 ~ \forall \alpha \in W_L\}.  
\end{equation}
Any $\beta \in L^n$ defines an $\mathcal{O}$-linear map 

 $$h_{\beta}: \left\{ \begin{array}{cc}
        \mathfrak{p}_L/W_L   & \longrightarrow  K/\mathcal{O} \\
       \alpha   & \longmapsto   \iota((\alpha,\beta)_{L,n})
    \end{array}
    \right.$$
where the action of $\mathcal{O}$ on $\mathfrak{p}_L/W_L$ is given by (\ref{defactionrho}). The map $\beta \mapsto h_{\beta}$ gives a group homomorphism from $L^n$ to $\operatorname{Hom}_{\mathcal{O}}(\mathfrak{p}_L/W_L, K/\mathcal{O})$. 
The isomorphism of Lemma \ref{lemmadefrho} (ii) induces the following isomorphism of $\mathcal{O}$-modules
\begin{equation} \label{isomhom}
    \operatorname{Hom}_{\mathcal{O}} (\mathfrak{p}_L/W_L  ,K/\mathcal{O} ) \cong \operatorname{Hom}_{\mathcal{O}} ( \lambda_{\rho}(\mathfrak{p}_L) ,K/\mathcal{O} ) .
\end{equation}

Let $\beta \in L^n$ and let $g_{\beta}$ be the image of $h_{\beta}$ by the isomorphism \eqref{isomhom}. Then $g_{\beta}$ is defined by $g_{\beta}(\lambda_{\rho}(\alpha))=\iota((\alpha,\beta)_{L,n})$. However $g_{\beta}$ is an $\mathcal{O}$-linear map from  $\lambda_{\rho}(\mathfrak{p}_L)$ to $K/\mathcal{O}$. Thus, by Lemma \ref{surjmap}, there exists a unique $y \in L/\mathfrak{X}_L$ satisfying $g_{\beta}(\lambda_{\rho}(\alpha))=\operatorname{T}_{L|K}(\lambda_{\rho}(\alpha)y) \mod \mathcal{O}$ for all $\alpha \in \mathfrak{p}_L$. It is easy to see that $y \in \eta^{-n}\mathfrak{X}_L/\mathfrak{X}_L$. We set
\begin{equation} \label{defPsi}
    \psi_{L,n}(\beta) = \eta^n y \mod \eta^n \mathfrak{X}_L.
\end{equation}

\begin{prop} \label{WilesProp}
We have
\begin{equation} 
    (\alpha,\beta)_{L,n}=  T_{L|K}(\lambda_{\rho}(\alpha)\psi_{L,n}(\beta)) \cdot_{\rho} v_n \label{formulapsi}
\end{equation}
for all $\alpha \in \mathfrak{p}_{L}$ and $\beta \in L^n$. Furthermore, the map $\psi_{L,n}:L^n \longrightarrow \mathfrak{X}_L/\eta^n\mathfrak{X}_L$ is a group homomorphism.
\end{prop}

\begin{proof}
The Proposition follows immediately from the construction.
\end{proof}

Exactly as in \cite{kolyvagin}, our $\psi_{L,n}$ satisfies the properties $\varphi_1,~ \varphi_2,~ \varphi_3,~ \varphi_4,~ \varphi_5$ and $\varphi_6$ of loc. cit.

\section{More properties of the pairing $(~,~)_{L,n}$}
As above, we continue to fix a positive integer $n$, a finite extension $L$ of $E_{\rho}^n$ and a generator $v_n$ of $W_{\rho}^n$.
\begin{lemma} \label{majoration}
The map $(~.~)_{L,n}$ is continuous, and $(\alpha,\cdot)_{L,n}=0$ for all $\alpha \in \mathfrak{p}_L$ such that $\mu(\alpha)> nm_0 + \dfrac{1}{q-1}$. Furthermore, if we put $\alpha_m= \rho_{\eta^{m-n}}(\alpha)$ for $m \ge n$, then there exists a constant $c(\alpha)$ satisfying $\mu(\alpha_m) \ge mm_0- c(\alpha)$.
\end{lemma} 

\begin{proof}
We follow \cite[Lemma 15]{Ignazio}. First we show that $(\alpha,\cdot)_{L,n}=0$ for all $\alpha \in \mathfrak{p}_L$ such that $\mu(\alpha)> nm_0 + \dfrac{1}{q-1}$. To do so, it is sufficient to chose $\xi$ a root of $\rho_{\eta ^n}(X)=\alpha$ such that $\mu(\xi)$ is large enough. Indeed, since $W_{\rho}^n$ is discrete and $\mu((\alpha,\beta)_{L,n}) \ge \mu(\xi)=\mu(\Phi_L(\beta)(\xi))$ for all $\beta \in L^{\times}$ , if $\mu(\xi)$ is large enough, then  $(\alpha,\beta)_n=0$. 

For the valuation computation, we set $\mu_j:=\displaystyle\frac{1}{q^{j-1}(q-1)}$ for $j \ge 1$ and $\mu_0:=\infty$. Choose $\xi$ a root of $\rho_{\eta ^n}(X)=\alpha$ of maximal valuation. This is possible because the equation $\rho_{\eta ^n}(X)=\alpha$ has a finite set of solutions: $\xi + W_{\rho}^n$. 
We have
$$ \alpha= \rho_{\eta^n}(\xi)  = \rho_{u^n}(P_{nm_0}(\xi)),$$
where $u$ is the unit such that $\eta= u \pi^{m_0}$ and $P_{nm_0}(X)=\Pi_{w \in W_{\rho}^n} (X-w)$ is the polynomial defined in \eqref{defPn}. Therefore, we get
$$ \mu(\alpha)= \mu(P_{nm_0}(\xi))=\sum_{w \in W_{\rho}^n} \mu(\xi -w)$$
because $\rho_{u^n}(X) \equiv u^n X \mod \operatorname{deg~2}$.
Let $w \in W_{\rho}^n$. If $\mu(\xi) \neq \mu(w)$, then $\mu(\xi-w)=\min\{\mu(\xi),\mu(w)\}$. If $\mu(\xi) = \mu(w)$, then 
$$ \mu(\xi) = \min\{\mu(\xi),\mu(w)\} \le \mu(\xi-w) \le \mu(\xi),$$
the last inequality being a consequence of the maximality hypothesis on $\mu(\xi)$. Hence we have $\mu(\xi-w)=\min\{\mu(\xi),\mu(w)\}$ for all $w \in W_{\rho}^n$ and
\begin{equation} \label{eqminval}
    \mu(\alpha) = \sum_{w \in W_{\rho}^n} \min\{\mu(\xi),\mu(w)\} .
\end{equation}
Let $j \ge 0$ be such that $v_{j+1} < \mu(\xi) \le v_j$. If $0 \le j \le nm_0$, the equality \eqref{eqminval} yields
$$ \mu(\alpha) = \sum_{w \in V_{\rho}^j} \mu(\xi) + \sum_{w \in W_{\rho}^n \setminus V_{\rho}^j} \mu(w) = q^j \mu(\xi) +nm_0-j  $$
so that $nm_0-j+\frac{1}{q-1} < \mu(\alpha) \le nm_0 -j + 1 +\frac{1}{q-1}.$  Now if $j > nm_0$, by \eqref{eqminval} we get
$ \mu(\alpha) = q^{nm_0} \mu(\xi)$ so that 
$$ nm_0-j+\frac{1}{q-1} \le 0 < \frac{1}{q^{j-nm_0}(q-1)} < \mu(\alpha) \le  \frac{1}{q^{j-nm_0-1}(q-1)}.$$
If we suppose $\mu(\alpha)> nm_0 + \dfrac{1}{q-1}$, we get $j=0$, which implies that $\mu((\alpha,\beta)_{L,n}) \ge \mu(\xi) > \frac{1}{q-1}$ for all $\beta \in L^{\times}$. It follows that $(\alpha,\beta)_{L,n}=0$ for all $\beta \in L^{\times}$. The fact that the map $(~.~)_{L,n}$ is continuous follows immediately since the reciprocity map $\Phi_L$ is continuous. Finally, we set $c(\alpha)= j - \frac{1}{q-1}$.
\end{proof}
\begin{remark} \label{remarkC}
Let $e$ be the ramification index of $L|E_{\rho}^n$, then the constant $c(\alpha)$ from the last Lemma is bounded as follows
$$ \dfrac{-1}{q-1} \le c(\alpha) \le 2nm_0 + \log_q(e) -\dfrac{1}{q-1}  .$$
\end{remark}
\begin{prop} \label{existencer}
There exists a unique power series $r=r_n \in \mathcal{O}_H\{\{\tau\}\}$ such that $$ \prod_{\omega \in W_{\rho}^n}(X-\omega)=r \circ \rho_{\eta^n}(X) .$$ 
Furthermore, the power series $r$ is invertible in $\mathcal{O}_H\{\{\tau\}\}$ and satisfies

$$(x,r(x))_{L,n}=0, ~~ \forall x \in \mathfrak{p}_L  .$$

\end{prop}

\begin{proof}
As in the proof of Lemma \ref{existenceXi}, we can write 
$$ \rho_{\eta^n} (X) = \rho_{u^n} \circ U_1 \circ U_{nm_0} \circ P_{nm_0}(X) .$$
Thus for $r= (\rho_{u^n} \circ U_1 \circ U_{nm_0})^{-1}$ we get $P_{nm_0}(X)= \prod_{\omega \in W_{\rho}^n}(X-\omega)= r \circ \rho_{\eta^n}(X)$.
It remains to show that $(x,r(x))_{L,n}=0$ for all $x \in \mathfrak{p}_L$. Take $x \in \mathfrak{p}_L$ and $\xi$ such that $\rho_{\eta^n}(\xi)=x$. Then,
$$ r(x)= (r \circ \rho_{\eta^n})(\xi) = \prod_{\omega \in W_{\rho}^n}(\xi-\omega) = \prod_{i} \operatorname{N}_{L(\xi)|L}(\xi_i)$$
where $\xi_i$ are the pairwise distinct roots of $\rho_{\eta^n}(X)=x$. It follows that $(x,r(x))_{L,n}=0$ by Proposition \ref{PropoertiesPairing1} (ii).
\end{proof}

\begin{lemma} \label{condrhoprime}
Let $\rho'$ be  defined by 
\begin{equation} \label{defrhoprime}
    \rho'_a= r \circ \rho_a \circ r^{-1}
\end{equation}
for all $a \in \mathcal{O}$. Then $\rho'$ is a formal Drinfeld module having a stable reduction of height 1, and we have $(x,x)_{\rho',L,n}=0$ for all $x \in \mathfrak{p}_L$.
\end{lemma}

\begin{proof}
That $\rho'$ is a formal Drinfeld module having a stable reduction of height 1 follows from the fact that $\rho$ itself is supposed to be a formal Drinfeld module having a stable reduction of height 1. It follows from Proposition \ref{PropoertiesPairing1} (vi) that $(x,x)_{\rho',L,n}= r((r^{-1}(x),x)_{\rho,L,n})=r(0)=0.$
\end{proof}

\begin{lemma} \label{lemmecb}
If $\rho$ is such that $(x,x)_{\rho,L,n}=0$ for all $x \in \mathfrak{p}_L$, then  we have 
$$ (c,1-b)_{L,n} = (\dfrac{bc}{1-b},b^{-1})_{L,n}$$
for all $b \in \mathfrak{p}_L \setminus\{0\}$ and $c \in \mathfrak{p}_L$.
\end{lemma}

\begin{proof}
We copy the proof of \cite[Lemma~18]{Ignazio}. To prove this result we use the property $(x,x)_{L,n}=0$ for $x=c(1-b)$. By bilinearity we get
$$ (c,1-b)_{L,n} = (cb,c)_{L,n} + (cb, 1-b)_{L,n} $$
and by induction
$$ (c,1-b)_{L,n}= \sum_{j \ge 0} (cb^j, cb^{j-1})_{L,n}. $$ 
This sum converges since only a finite number of terms is non zero by Lemma \ref{majoration}. However we have 
$ 0= (cb^j,cb^j)_{L,n} =(cb^j,cb^{j-1})_{L,n} + (cb^j,b)_{L,n}  .$
Therefore, 
$$ (c,1-b)_{L,n}= \sum_{j \ge 0} (cb^j, b^{-1})_{L,n} = (c \sum_{j \ge 0} b^j, b^{-1})_{L,n}  = (\frac{cb}{1-b}, b^{-1})_{L,n} . $$ 
\end{proof}
For a finite extension $F'|F$ of local fields, let $\mathfrak{m}_{F'|F}$ be the fractional ideal of $\mathcal{O}_{F'}$ defined by 
$$ \mathfrak{m}_{F'|F}=\{x \in F';~ \operatorname{T}_{F'|F}(x\mathcal{O}_{F'})\subset\mathcal{O}_F\} \supset \mathcal{O}_{F'}$$
\noindent As usual, the different $\mathcal{D}_{F'|F}$ of $F'|F$ is the inverse ideal of $\mathfrak{m}_{F'|F}$
$$ \mathcal{D}_{F'|F} := \mathfrak{m}_{F'|F}^{-1}.$$
If $F'|F$ is unramified, then $\mathcal{D}_{F'|F}= \mathcal{O}_{F'}$, and if $F'|F$ is totally ramified, then $\mathcal{D}_{F'|F}=h'(w)\mathcal{O}_{F'}$, where $w$ is a prime element of $F'$ and $h(X)$ is the minimal polynomial of $w$ over $F$. Moreover, if $F''|F$ is a finite extension of local fields such that $F \subset F' \subset F''$, we have 
$$ \mathcal{D}_{F''|F}=\mathcal{D}_{F''|F'} \mathcal{D}_{F'|F}.$$
 For more details, the reader may check \cite[\S 2.4]{Iwasawa}.

\begin{lemma} \label{generatordifferent}
Let $\mathcal{D}_n$ be the different of the extension $E_{\rho}^n|K$, then $\mathcal{D}_n$  is generated by an element of valuation 
 $nm_0 - \dfrac{1}{q-1} .$
\end{lemma}

\begin{proof}
First, note that we have $\mathcal{D}_n=\mathcal{D}_{E_{\rho}^n|H} \mathcal{D}_{H|K}$ . Since $H|K$ is an unramified extension, then $\mathcal{D}_{H|K}= \mathcal{O}_H$. Since $E_{\rho}^n|H$ is totally ramified and $v_n$ is a prime element of $E_{\rho}^n$, then $\mathcal{D}_{E_{\rho}^n|H}$ is the ideal of $\mathcal{O}_n$ generated by $h_n'(v_n)$, where $h_n$ is the minimal polynomial of $v_n$ over $H$. By \cite[Section 2]{Oukhaba}, $h_n(X)= \Pi_w (X-w)$ where $w$ varies in $V_{\rho}^{nm_0} \setminus V_{\rho}^{nm_0-1}$. We can also write 
\begin{equation} \label{phproduct}
    h_n(X)P_{nm_0-1}(X)=P_{nm_0}(X)
\end{equation}
where $P_{l}= \Pi_{w \in V_{\rho}^l} (X-w)$ for any positive integer $l$. Differentiating \eqref{phproduct} and evaluating at $v_n$ we get $h_n'(v_n)P_{nm_0-1}(v_n)=P'_{nm_0}(v_n)$. Since $\mu(v_n)=\frac{1}{q^{nm_0-1}(q-1)} < \mu(w)$ for all $w \in V_{\rho}^{nm_0-1}$, then $$\mu(P_{nm_0-1}(v_n))= \sum_{w\in V_{\rho}^{nm_0-1}}(\mu(v_n))=\frac{1}{q-1}.$$
Moreover, we have  $P_{nm_0} \in \mathcal{O}_H\{\tau\}$ as in \eqref{defPn} so that $\mu(P_{nm_0}(v_n))=nm_0$. This concludes the proof.
\end{proof}

\begin{lemma} \label{different}
Let $x \in E_{\rho}^n$ and denote by $\operatorname{T}_{n}$ the trace of the extension $E_{\rho}^n|K$. Then,
$$ \mu(\operatorname{T}_{n}(x)) \ge \lfloor\mu(x) + n m_0 - \dfrac{1}{q-1} \rfloor, $$
where $\lfloor a \rfloor$ the integral part of $a \in \mathbb{R}$. Furthermore, for $m \le n$, we have
$$ \mu(\operatorname{T}_{n,m}(x)) > \mu(x) +  (n-m)m_0 - \mu(v_{m}),$$
where $\operatorname{T}_{n,m}$ is the trace of the extension $E_{\rho}^n|E_{\rho}^{m}$.
\end{lemma}

\begin{proof}
Let $k=\lfloor \mu(x) + m m_0 - \dfrac{1}{q-1} \rfloor $ then $x \mathcal{O}_n \subset \mathfrak{p}^k \mathcal{D}_n^{-1}$. Thus $\operatorname{T}_n(x \mathcal{O}_n) \subset \mathfrak{p}^k \mathcal{O} $. In the same way, the generator on $\mathcal{D}_{E_{\rho}^n|E_{\rho}^{m}}$ is of valuation $(n-m)m_0$. For $k=\lfloor \displaystyle\frac{\mu(x)+(n-m)m_0}{\mu(v_{m})} \rfloor$, we have $\operatorname{T}_{n,m}(x\mathcal{O}_n) \subset \mathfrak{p}_{m}^k \mathcal{O}$. Thus, we have
$$ \mu(\operatorname{T}_{n,m}(x)) \ge k \mu(v_{m}) > \mu(x) + (n-m)m_0 - \mu(v_{m}). $$
\end{proof}
For the rest of the paper, we suppose $L=E_{\rho}^n$. 
\begin{lemma} \label{propertiesDelta} 
\begin{enumerate} [label=(\roman*)]
    \item The map $\delta_n: L^{\times} \rightarrow \mathfrak{p}_L^{-1}/ \mathcal{D}_n$ defined by 
    \begin{equation}
    \delta_n(\beta):= \dfrac{f'(v_n)}{\beta} \mod \mathcal{D}_n,
    \end{equation} 
where $f \in \mathcal{O}_H((X))^{\times}$ is such that $f(v_n)=\beta$, is a group homomorphism.
    \item For $m \ge n$ and $\beta \in L^{\times}$, we have 
    $$ \delta_m(\beta) \equiv \eta^{m-n} \delta_n(\beta) \mod \mathcal{D}_m.$$
\end{enumerate}

\end{lemma}

\begin{proof}
This lemma is easy to prove, the interested reader may check \cite[Lemma 10]{Wiles}.
\end{proof}

\begin{lemma} \label{isdefinedP2}
The map  
$$ [\alpha, \beta]_{\rho,L,n} := \frac{1}{\eta^n} \operatorname{T}_{L|K} (\lambda_{\rho}(\alpha) \delta_n(\beta)) \cdot_{\rho} v_n $$ 
is well defined for all $\alpha \in \mathfrak{p}_L$ of valuation $\mu(\alpha) \ge \frac{2}{q-1}$, and all $\beta \in L^{\times}$. We drop $\rho$ in the index when there is no risk of confusion.
\end{lemma}

\begin{proof}
We need to show that $\frac{1}{\eta^n} \operatorname{T}_{L|K} (\lambda_{\rho}(\alpha) b) \in \mathcal{O}$ for every $b \in \mathfrak{p}_L^{-1}$ and that 
$$ \mu( \frac{1}{\eta^n} \operatorname{T}_{L|K} (\lambda_{\rho}(\alpha) d)) \ge nm_0 $$ for all $d \in \mathcal{D}_n$.
Using \ref{coefflamda} of Lemma \ref{lemmadefrho}, we can deduce that $\mu(\lambda_{\rho}(\alpha))= \mu(\alpha)$ . Thus the result follows from Lemma \ref{different}.
\end{proof}
\begin{prop} \label{PropertiesPairing2}
The map $[~,~]_{L,n}$ satisfies the following properties

\begin{enumerate} [label=(\roman*)]
    \item The map $[~,~]_{L,n}$ is bilinear and $\mathcal{O}$-linear in the first coordinate for the action (\ref{defactionrho}).
    \item Let $\rho'$ be a formal Drinfeld $\mathcal{O}$-module isomorphic to $\rho$, i.e there exists a power series $t$ invertible in $\mathcal{O}_H\{\{\tau\}\}$ such that $\rho'_a= t^{-1} \circ \rho_a \circ t$ for all $a \in \mathcal{O}$. Then we have $[\alpha,\beta]_{\rho',L,n}= t^{-1}([t(\alpha),\beta]_{\rho,L,n})$.
\end{enumerate}
\end{prop}

\begin{proof}
The property (i) is clear, so we will only prove (ii). To do so, let $v'_n = t^{-1}(v_n)$ be a generator of the $\mathcal{O}$-module $W_{\rho'}^n$. Then, if $f \in \mathcal{O}_H((X))^{\times}$ is such that $f(v_n)=\beta$, we have $f \circ t (v'_n) = f(v_n) = \beta$ so that

$$ \delta'_n(\beta) = \dfrac{t'(v'_n) f'(v_n)}{\beta}  = t'(0) \delta_n(\beta), $$
where $\delta'_n$ is the map defined in Lemma \ref{propertiesDelta} corresponding to $\rho'$. Furthermore, we have $\lambda_{\rho'} \circ t^{-1}= (t^{-1})'(0) \lambda_{\rho} $. The result follows immediately since $(t^{-1})'(0) = \dfrac{1}{t'(0)}$.
\end{proof}
\begin{lemma} \label{removelambda}
Let $\alpha \in \mathfrak{p}_L$ such that $\mu(\alpha) \ge \frac{nm_0}{q}+\frac{1}{q-1}+ \frac{1}{q^{nm_0}(q-1)}$ and let $\beta \in L^{\times}$. We have
$$  [\alpha, \beta]_{L,n} = \frac{1}{\eta^n} \operatorname{T}_{L|K} (\alpha \delta_n(\beta)) \cdot_{\rho} v_n .$$
\end{lemma}

\begin{proof}
We need to prove that 
$$\frac{1}{\eta^n} \operatorname{T}_{L|K} (\lambda_{\rho}(\alpha) \delta_n(\beta)) \cdot_{\rho} v_n  =  \frac{1}{\eta^n} \operatorname{T}_{L|K} (\alpha \delta_n(\beta)) \cdot_{\rho} v_n ,$$
i.e. that 
$$ \mu(\operatorname{T}_{L_K}(\lambda_{\rho}(\alpha)-\alpha)\delta_n(\beta)) \ge 2nm_0 .$$
We have 
$$ \mu((\lambda_{\rho}(\alpha)-\alpha) \delta_n(\beta)) \ge \min_i\{q^i\mu(\alpha)-i\} - \frac{1}{q^{nm_0-1}(q-1)}. $$
The hypothesis implies that $\min_i\{q^i\mu(\alpha)-i\} = q\mu(\alpha)-1$ so that $\mu(\lambda_{\rho}(\alpha)-\alpha)\delta_n(\beta)) \ge nm_0 + \frac{1}{q-1}$. Finally, we conclude using Lemma \ref{different}.
\end{proof}
\begin{lemma} \label{deltamn1}
Suppose $\rho$ is such that $ \rho_{\eta} \equiv \tau^{m_0} \mod \mathfrak{p}_H$. Let $\beta \in E_{\rho}^n$ and $\beta' \in E_{\rho}^m$ such that $\operatorname{N}_{m,n}(\beta')=\beta$. We have
\begin{equation*}
    \operatorname{T}_{m,n}(\delta_m(\beta'))= \eta^{m-n} \delta_n(\beta).
\end{equation*}
\end{lemma}
\begin{proof}
This lemma is the analogue of Lemma 8.9 in \cite{Iwasawa}, whose proof is adaptable to our case. The main ingredient used is the Coleman norm operator associated to $\rho$, defined by Oukhaba in \cite[\S 5]{Oukhaba}.
\end{proof}

\section{Explicit reciprocity laws}
In this section, we assume that $ \rho_{\eta} \equiv \tau^{m_0} \mod \mathfrak{p}_H$. We fix a positive integer $n$ and a generator $v_n$ of $W_{\rho}^n$, and we set $L=E_{\rho}^n$.
As in the classical case of Lubin-Tate formal groups, we have
\begin{prop}

 \label{rho=phi-1}
For every unit $u$ of $K$, we have
\begin{equation} 
    \Phi_K(u)(\omega)= \rho_{u^{-1}}(\omega)
\end{equation}

\noindent for all $\omega \in W_{\rho}$.
\end{prop}


\begin{proof}
Let $f(X)=\pi X + X^q$. By Lubin-Tate theory, there exists an injective ring homomorphism $\mathcal{O} \longrightarrow \operatorname{End}(G_a)$ which associates for each $a \in \mathcal{O}$ a unique power series $[a]_f$ such that 
$$ [a]_f(X) \equiv aX \mod {deg~2} ~~ \text{and} ~~ f \circ [a]_f = [a]_f^{\phi} \circ f,$$
where $\phi$ is the Frobenius automorphism of $K_{ur}|K$.
Clearly, we have $f(X)=[\pi]_f(X)$. Since 
$$ \rho_{\eta}(X) \equiv \eta X,~~ [\pi^{m_0}]_f(X) \equiv \pi^{m_0}X \mod deg~2 ~~ \text{and}~~ \rho_{\eta}(X) \equiv [\pi^{m_0}]_f(X) \equiv X^{q^{m_0}} \mod \mathfrak{p}_H,$$
by \cite[Proposition~3.1]{Oukhaba}, there exists a unique power series $\theta \in \mathcal{O}_{\Bar{K}_{ur}}[[X]]$ such that 
$$ \theta(X) \equiv u_0X \mod deg~2 ~~ \text{and} ~~ \rho_{\eta} \circ \theta = \theta^{\phi^{m_0}} \circ [\pi^{m_0}]_f $$ where $u_0$ is the unit of $K$ such that $\eta = u_0 \pi^{m_0}$. We deduce that for all $m \ge 1$, we have 
$$ \rho_{\eta^m} \circ \theta = \theta^{\phi^{m_0}} \circ [\pi^{mm_0}]_f $$
and therefore the  isomorphism 
$$ \theta: W_f \longrightarrow W_{\rho}.$$
Here $W_f = \bigcup W_f^m$, where $W_f^m$ is the set of roots of $[\pi^m]_f$.
Now let $u$ be a unit of $K$ and consider $\Phi_K(u) \in Gal(K^{ab}|K_{ur})$. By \cite[Chapter~6]{Iwasawa}, we have
$$ \Phi_K(u)(\omega')= [u^{-1}]_f (\omega') ~~ \forall \omega' \in W_f .$$
However, since $\Phi_K(u)_{|_{H(W_{\rho})}} \in Gal(H(W_{\rho})|H)$, by \cite[Proposition~2.5]{Oukhaba} we have 
$$ \Phi_K(u)(\omega) = \rho_{{v}^{-1}}(\omega), ~~ \forall \omega \in W_{\rho}$$
for a unit $v \in K$. Let $\omega' \in W_f$ and $\omega= \theta(\omega') \in W_{\rho}$. We have 
\begin{align*}
    \rho_{{v}^{-1}} \circ \theta(\omega') = \rho_{{v}^{-1}}(\omega)= \Phi_K(w)(\omega)
    & = \Phi_K(u)(\theta(\omega')) \\
    & = \theta (\Phi_K(u)(\omega')) \\
    & = \theta \circ [u^{-1}]_f (\omega').
\end{align*}
This is true for all $\omega' \in W_f$. Then, reasoning as in \cite[Lemma 8.1]{Iwasawa}, we can prove that
$$ \rho_{v^{-1}} \circ \theta = \theta \circ [u^{-1}]_f. $$
By identification, we get $u=v$. This concludes the proof.
\end{proof}


\begin{lemma} \label{equality}
Suppose that for any $A >0$, there exists $m \ge A$ such that $(x,x)_{E_{\rho}^m,m}=0$ for all $x \in \mathfrak{p}_m$. Then, there exists a prime $\pi_n$ of $L$ such that
$$(\alpha,\pi_n)_{L,n}=[\alpha,\pi_n]_{L,n}=\frac{1}{\eta^n}\operatorname{T}_{L|K}(\lambda_{\rho}(\alpha)\delta_n(\pi_n))\cdot_{\rho} v_n$$
for all $\alpha \in \mathfrak{p}_L$.
\end{lemma}

\begin{proof}
We prove the Lemma following the steps of \cite[Proposition 23]{Ignazio}, which were essentially used by Wiles \cite[Lemma 8]{Wiles}. Let $\alpha \in \mathfrak{p}_L$.

\noindent \textbf{Step 1:} For $m \ge n$, let $\alpha_m=\rho_{\eta^{m-n}}(\alpha)$, let  $v_m$ be a generator of the $\mathcal{O}$-module $W_{\rho}^m$, and $b_m=\alpha_m v_m^{-1}$. 
If we suppose $(x,x)_{E_{\rho}^m,m}=0$ for all $x \in \mathfrak{p}_m$, we have
$$ 0= (\alpha_m + v_{m}, (1+b_m)v_m)_{E_{\rho}^m,m} = (\alpha_m,v_m)_{E_{\rho}^m,m}+ (\alpha_m, 1+b_m)_{E_{\rho}^m,m} + (v_m, 1+b_m)_{E_{\rho}^m,m},$$
because $\alpha_m +v_m = (1+b_m)v_m$.

\noindent \textbf{Step 2:} For all sufficiently large $m$, we have $(\alpha_m, 1+b_m)_{E_{\rho}^m,m}=0$, for all $\alpha \in \mathfrak{p}_L$. Indeed, 
$$ (\alpha_m, 1+b_m)_{E_{\rho}^m,m}= (\alpha, \operatorname{N}_{m,n}(1+b_m))_{L,n} ,$$
where $\operatorname{N}_{m,n}$ is the norm of $E_{\rho}^m|E_{\rho}^n$. However, by Lemma \ref{majoration} and Remark \ref{remarkC}, there exists a constant $c$ depending only on $n$ such that $\mu(b_m) \ge mm_0-c$ for all $m\ge n$. Thus $1+b_m$ tends to $1$, and so does $\operatorname{N}_{m,n}(1+b_m)$. Since $(\alpha,\cdot)_{L,n}$ is continuous and $W_{\rho}^n$ is discrete, the result follows.

\noindent \textbf{Step 3:} Let $m \gg n$ so that $(\alpha_m, 1+b_m)_{E_{\rho}^m,m}=0$ and suppose that $(x,x)_{E_{\rho}^m,m}=0$ for all $x \in \mathfrak{p}_m$. Let $\pi_n = \operatorname{N}_{m,n}(v_m)$, then $\pi_n$ is a prime of $L$ because $E_{\rho}^m|L$ is a totally ramified extension. Let $v_{2m}$ be a generator of $W_{\rho}^{2m}$ such that $\rho_{\eta^m}(v_{2m})=v_m$. We have 
$$(\alpha,\pi_n)_{L,n}= v_{2m} - \rho_{\operatorname{N}_{m}(1+b_m)^{-1}}(v_{2m}),$$
where $\operatorname{N}_m$ is the norm of $E_{\rho}^m|K$. Indeed,
\begin{align*}
    (\alpha,\pi_n)_{L,n}= (\alpha_m,v_m)_{E_{\rho}^m,m} 
    & = -(v_m, 1+b_m)_{E_{\rho}^m,m} \quad \quad  \quad   \text{(by Step 1 and 2)}\\
    & = - (\Phi_{E_{\rho}^m}(1+b_m)(v_{2m})-v_{2m})\\
    & = - (\Phi_K(\operatorname{N}_m(1+b_m))(v_{2m})- v_{2m}).
\end{align*}
By Proposition \ref{rho=phi-1}  we have $\Phi_K(\operatorname{N}_m(1+b_m))(v_{2m})= \rho_{\operatorname{N}_{m}(1+b_m)^{-1}}(v_{2m})$ and hence $(\alpha,\pi_n)_n= v_{2m} - \rho_{\operatorname{N}_{m}(1+b_m)^{-1}}(v_{2m})$.

\noindent \textbf{Step 4:} $\operatorname{N}_m(1+b_m)^{-1} \equiv 1- \operatorname{T}_m (b_m) \mod \mathfrak{p}^{2mm_0}$ for all sufficiently large $m$. Indeed, let $k$ be a positive integer such that $km_0 \ge c+1$ and let $m \gg k$. Let $x= \operatorname{T}_{m,m-k}(b_m)$, thus 
$$ \operatorname{N}_{m,m-k}(1+b_m)^{-1}= 1-x+y, $$
where $\mu(y) \ge 2 \mu(b_m)$. Therefore, we have
$$ \operatorname{N}_m(1+b_m)^{-1}= \operatorname{N}_{m-k}(1-x+y) = 1 -\operatorname{T}_{m-k}(x-y) + z \equiv 1 - \operatorname{T}_m(b_m) \mod \mathfrak{p}^{2mm_0} ,$$
because $\mu(T_{m-k}(y)) \ge 2mm_0$ and  $\mu(z) \ge 2 \mu(x-y) \ge 2mm_0$ by Lemma \ref{different}.

\noindent \textbf{Step 5:} Choose $m \ge \frac{q}{q-1}(2n+\frac{1}{2m_0})$, and sufficiently large to satisfy Step 2 and Step 4. If in addition we have $(x,x)_{E_{\rho}^m,m}=0$ for all $x \in \mathfrak{p}_m$, then $(\alpha, \pi_n)_{L,n}=[\alpha, \pi_n]_{L,n}$, where $\pi_n = \operatorname{N}_{m,n}(v_m)$ as in Step 3. 
Indeed, by the previous steps we get  $(\alpha, \pi_n)_{L,n}=\operatorname{T}_m(\alpha_m v_m^{-1}) \cdot _{\rho} v_{2m} =\dfrac{1}{\eta^m}\operatorname{T}_m(\alpha_m v_m^{-1}) \cdot_{\rho} v_m $. 
Moreover, $m$ is sufficiently large so that $\mu(\alpha_m) \ge \frac{mm_0}{q}+\frac{1}{q-1}+ \frac{1}{q^{mm_0}(q-1)}$. Therefore, by Lemma \ref{removelambda}, we get 
\begin{equation*}
(\alpha, \pi_n)_{L,n} =\dfrac{1}{\eta^m}\operatorname{T}_m(\alpha_m v_m^{-1}) \cdot_{\rho} v_m = [\alpha_m,v_m]_{E_{\rho}^m,m} =  [\alpha, \pi_n]_{L,n}.
\end{equation*}

\end{proof}
\begin{lemma} \label{lemmaunit}
Suppose $\rho$ is such that $(x,x)_{L,n}=0$ for all $x \in \mathfrak{p}_L$. Let $\alpha \in \mathfrak{p}_L$ such that $\mu(\alpha) \ge \frac{nm_0}{q} + \frac{1}{q-1} + \frac{1}{q^{nm_0}(q-1)}$ and $\beta$ a unit in $L^{\times}$. Then 
$$ (\alpha,\beta)_{L,n} =[\alpha,\beta]_{L,n} =  \frac{1}{\eta^n}\operatorname{T}_{L|K}(\lambda_{\rho}(\alpha)\delta_n(\beta))\cdot_{\rho} v_n.$$
\end{lemma}

\begin{proof}

For all $m \ge n$, Proposition \ref{existencer} shows that there exists an invertible power series $r \in \mathcal{O}_H\{\{\tau\}\}$ such that $ \prod_{\omega \in W_{\rho}^m}(X-\omega)=r \circ \rho_{\eta^m}(X)$. Then, by Lemma \ref{condrhoprime}, the formal Drinfeld module $\rho'$ defined by $ \rho'_a= r \circ \rho_a \circ r^{-1}$ for all $a \in \mathcal{O}$ satisfies $(x,x)_{\rho',E_{\rho}^m,m}=0$.
Choose $m$ sufficiently large so that Lemma \ref{equality} is satisfied for $\rho'$, and let $\pi'_n$ be a prime of $L$ satisfying $(\alpha,\pi'_n)_{\rho',L,n}=[\alpha,\pi'_n]_{\rho',L,n}$. 
We notice that it is sufficient to prove the Lemma for $\beta= 1- \zeta {\pi'_n}^j$, where  $\zeta$ is a $(q-1)^{\text{th}}$ root of unity, and $j \ge 1$ is an integer. By Lemma \ref{lemmecb}, we have
\begin{align}
    (\alpha,1-\zeta {\pi'_n}^j)_{\rho,L,n} & = (\dfrac{\zeta {\pi'_n}^j}{1-\zeta {\pi'_n}^j} \alpha, (\zeta {\pi'_n}^j)^{-1})_{\rho,L,n} \\
    & = -j (\dfrac{\zeta {\pi'_n}^j}{1-\zeta {\pi'_n}^j} \alpha, \pi'_n)_{\rho,L,n} . \label{equnit1}
\end{align}
Using Proposition \ref{PropoertiesPairing1} (vi) and then Lemma \ref{equality} for $\rho'$, \eqref{equnit1} is equal to

\begin{equation}
     -j r^{-1}((r(\dfrac{\zeta {\pi'_n}^j}{1-\zeta {\pi'_n}^j} \alpha),\pi'_n)_{\rho',L,n}) 
     = -j r^{-1}([r(\dfrac{\zeta {\pi'_n}^j}{1-\zeta {\pi'_n}^j} \alpha),\pi'_n]_{\rho',L,n}). \label{equnit2}
\end{equation}
By Proposition \ref{PropertiesPairing2} (ii), \eqref{equnit2} is equal to 
\begin{align}
    -j [\dfrac{\zeta {\pi'_n}^j}{1-\zeta {\pi'_n}^j} \alpha,\pi'_n]_{\rho,L,n} 
    & =  \frac{-j}{\eta^n} \operatorname{T}_{L|K} (\dfrac{\zeta {\pi'_n}^j}{1-\zeta {\pi'_n}^j} \times \alpha \times \delta_n(\pi'_n)) \cdot_{\rho} v_n \label{equnit3} \\
    & =  \frac{1}{\eta^n} \operatorname{T}_{L|K} (\dfrac{-j \zeta {\pi'_n}^j}{1-\zeta {\pi'_n}^j} \times \alpha \times \frac{t'(v_n)}{\pi'_n}) \cdot_{\rho} v_n \label{equnit5}
\end{align}
where \eqref{equnit3} is deduced from Lemma \ref{removelambda}, and $t(X) \in \mathcal{O}_H((X))$ satisfying $t(v_n)=\pi'_n$. Since $1- \zeta ({t(v_n)})^j= 1- \zeta {\pi'_n}^j $, we have $$\delta_n(1- \zeta {\pi'_n}^j) = \dfrac{-j\zeta {\pi'_n}^{j-1} t'(v_n) }{1-\zeta {\pi'_n}^j},$$ and thus by Lemma \ref{removelambda}, \eqref{equnit5} is equal to $\frac{1}{\eta^n}\operatorname{T}_{L|K}(\lambda_{\rho}(\alpha)\delta_n(1- \zeta {\pi'_n}^j))\cdot_{\rho} v_n$. Hence 

\begin{equation*}
     (\alpha,1-\zeta {\pi'_n}^j)_{\rho,L,n} = [\alpha, 1- \zeta {\pi'_n}^j]_{\rho,L,n}.
\end{equation*}

\end{proof}

\begin{prop} \label{propunit}
Let $\alpha \in \mathfrak{p}_L$ such that $\mu(\alpha) \ge \frac{nm_0}{q} + \frac{1}{q-1} + \frac{1}{q^{nm_0}(q-1)}$ and $\beta$ a unit in $L^{\times}$. Then 
$$ (\alpha,\beta)_{L,n} =[\alpha,\beta]_{L,n} =  \frac{1}{\eta^n}\operatorname{T}_{L|K}(\lambda_{\rho}(\alpha)\delta_n(\beta))\cdot_{\rho} v_n.$$
\end{prop}

\begin{proof}
By Proposition \ref{existencer}, there exists an invertible power series $r \in \mathcal{O}_H\{\{\tau\}\}$ such that $ \prod_{\omega \in W_{\rho}^n}(X-\omega)=r \circ \rho_{\eta^n}(X)$.
Let $\rho'$ be the formal Drinfeld module defined by $ \rho'_a= r \circ \rho_a \circ r^{-1}$ for all $a \in \mathcal{O}$. Then, by Lemma \ref{condrhoprime} we have $(x,x)_{\rho',E_{\rho}^n,n}=0$. Hence, by Lemma \ref{lemmaunit} for $\rho'$, we have 
\begin{equation*}
    (\alpha,\beta)_{\rho,L,n} = r^{-1} ( (r(\alpha),\beta)_{\rho',L,n}) = r^{-1} ( [r(\alpha),\beta]_{\rho',L,n}) = [\alpha,\beta]_{\rho,L,n}.
\end{equation*}
\end{proof}
\begin{prop} \label{resforVn}
Let $\alpha \in \mathfrak{p}_L$  
such that $\mu(\alpha) \ge \frac{nm_0}{q} + \frac{1}{q-1} + \frac{1}{q^{nm_0}(q-1)}$ and let $\beta$ be a prime of $L$, then
$$ (\alpha,\beta)_{L,n} =[\alpha,\beta]_{L,n} =  \frac{1}{\eta^n}\operatorname{T}_{L|K}(\lambda_{\rho}(\alpha)\delta_n(\beta))\cdot_{\rho} v_n.$$
\end{prop}

\begin{proof}
Again by Proposition \ref{existencer}, for each $m \ge n$, there exists $r \in \mathcal{O}_H\{\{\tau\}\}$ be such that $ \prod_{\omega \in W_{\rho}^m}(X-\omega)=r \circ \rho_{\eta^m}(X)$.
Let $\rho'$ be the formal Drinfeld module defined by $ \rho'_a= r \circ \rho_a \circ r^{-1}$ for all $a \in \mathcal{O}$. Thus by Lemma \ref{condrhoprime}, we have $(x,x)_{\rho',E_{\rho}^m,m}=0$. Choose $m$ sufficiently large so that Lemma \ref{equality} is satisfied for $\rho'$ and let $\pi'_n$ be a prime of $E_{\rho}^n$ satisfying $(\alpha,\pi'_n)_{\rho',L,n}=[\alpha,\pi'_n]_{\rho',L,n}$. Then we can write $\beta = u \pi'_n$ for a unit $u \in L$. Hence,
\begin{equation*}
    (\alpha,\beta)_{\rho,L,n} = (\alpha,u \pi'_n)_{\rho,L,n} = (\alpha,u)_{\rho,L,n} +(\alpha,\pi'_n)_{\rho,L,n} .
\end{equation*}
By Proposition \ref{propunit}, we have $(\alpha,u)_{\rho,L,n} = [\alpha,u]_{\rho,L,n}$. On the other hand, by Proposition \ref{PropoertiesPairing1} (vi), we have 
\begin{equation*}
   (\alpha,\pi'_n)_{\rho,L,n} = r^{-1}((r(\alpha),\pi'_n)_{\rho',L,n}) = r^{-1}([r(\alpha),\pi'_n]_{\rho',L,n}),
\end{equation*}
the last equality being deduced from Lemma \ref{equality} for $\rho'$. Hence, by Proposition \ref{PropertiesPairing2}, we have
$$ (\alpha,\beta)_{\rho,L,n} =  [\alpha,u]_{\rho,L,n} + [\alpha,\pi'_n]_{\rho,L,n}  = [\alpha,\beta]_{\rho,L,n}.$$
\end{proof}
Combining Proposition \ref{propunit} and Proposition \ref{resforVn}, we obtain

\begin{theorem} \label{mainThm}
Let $\alpha \in \mathfrak{p}_L$ such that $\mu(\alpha) \ge \frac{nm_0}{q} + \frac{1}{q-1} + \frac{1}{q^{nm_0}(q-1)}$ and $\beta \in L^{\times}$. We have
$$
    (\alpha,\beta)_{\rho,L,n}= [\alpha,\beta]_{\rho,L,n}=\frac{1}{\eta^n} T_{E_{\rho}^n|K}(\lambda_{\rho}(\alpha)\delta_n(\beta)) \cdot_{\rho} v_n.
$$
\end{theorem}

\bibliographystyle{unsrt}

\begin{thebibliography}{1}
\bibitem{Ignazio}
Bars, F., Longhi, I.: Coleman's power series and {W}iles' reciprocity for rank 1 Drinfeld modules. J. Number Theory. 129(4), 789--805 (2009)
              
\bibitem{Oukhaba}
Oukhaba, H.: On local fields generated by division values of formal Drinfeld modules. Glasg. Math. J. 62(2), 459--472 (2020)

\bibitem{Rosen}
Rosen, M.: Formal {D}rinfeld modules. J. Number Theory, 103(2), 234--256 (2003)

\bibitem{Iwasawa}
Iwasawa, K.: Local class field theory. Oxford Science Publications, Oxford Mathematical Monographs. The Clarendon Press, Oxford University Press, New York (1986)

\bibitem{Drinfeld}
Drinfeld, V. G.: Elliptic modules. Mat. Sb. (N.S.), 94(136), 594--627, 656 (1974)

\bibitem{Iwasawa2}
Iwasawa, K.: On some modules in the theory of cyclotomic fields. J. Math. Soc. Japan, 16, 42--82 (1964)

\bibitem{Wiles}
Wiles, A.: Higher explicit reciprocity laws. Ann. of Math. (2), 107(2), 235--254 (1978)

\bibitem{kolyvagin}
Kolyvagin, V. A.: Formal groups and the norm residue symbol. Izv. Akad. Nauk SSSR Ser. Mat. 43(5), 1054--1120, 1198 (1979)

\bibitem{coleman}
Coleman, R. F.: Division values in local fields. Invent. Math. 53(2), 91--116  (1979)

\bibitem{Artin&Hasse}
Artin, E., Hasse, H.: Die beiden {E}rg\"{a}nzungss\"{a}tze zum reziprozit\"{a}tsgesetz der {$l^n$}-ten potenzreste im k\"{o}rper der {$l^n$}-ten {E}inheitswurzeln. Abh. Math. Sem. Univ. Hamburg, 6(1), 146--162 (1928)

\bibitem{angles}
Angl\`es, B.: On explicit reciprocity laws for the local {C}arlitz-{K}ummer symbols. J. Number Theory, 78(2), 228--252 (1999)

\bibitem{Lang}
Lang, S.: Algebra. Graduate Texts in Mathematics, 211 (third edition). Springer-Verlag, New York (2002)

\bibitem{Florez1}
Fl\'{o}rez, J.: Explicit reciprocity laws for higher local fields. J. Number Theory, 213, 400--444 (2020)

\bibitem{Florez2}
Fl\'{o}rez, J.: The norm residue symbol for higher local fields. J. Number Theory, in press (2021)

\end{thebibliography}



\end{document}